\documentclass[10pt]{amsart}

\usepackage[latin1]{inputenc}
\usepackage[T1]{fontenc}
\usepackage{amsmath}
\usepackage[initials,backrefs]{amsrefs}
\usepackage{amsthm}
\usepackage{latexsym}
\usepackage{amssymb}
\usepackage{mathtools}

\mathtoolsset{showonlyrefs} 
\usepackage[all]{xy}
\usepackage{calc}
\usepackage{multicol}
\usepackage{color}
\usepackage[left=2.61cm,right=2.61cm,top=2.72cm,bottom=2.72cm]{geometry}
\usepackage{ulem}
\usepackage{hyperref}
\newtheorem{thm}{Theorem}[section]
\newtheorem{prop}[thm]{Proposition}

\newtheorem{lem}[thm]{Lemma}

\newtheorem{rmq}[thm]{Remark}

\newcommand{\R}{\mathbb{R}}
\numberwithin{equation}{section}
\newcommand{\N}{\mathbb{N}}

\newcommand{\Bcal}{\mathcal{B}}

\newcounter{exercice}

\begin{document}

\title[Singular radial solutions for LNT]{Singular radial solutions for the Lin-Ni-Takagi equation }

\thanks{ J. F\" oldes is partly supported by the National Sicence Foundation under the grant NSF-DMS-1816408. 
\\
$^*$ corresponding author}

\author[Casteras, and F\" oldes]{ Jean-Baptiste Casteras \and Juraj F\" oldes$^*$}

\address{ Jean-Baptiste Casteras
\newline \indent Dept. of Mathematics and Statistics, University of Helsinki
\newline \indent P.O. Box 68 (Pietari Kalmin katu 5), FI-00014, Finland,\indent}
\email{jeanbaptiste.casteras@gmail.com}

\address{Juraj F\" oldes
\newline \indent Dept. of Mathematics, University of Virginia \indent 
\newline \indent  322 Kerchof Hall, Charlottesville, VA 22904-4137,\indent }
\email{foldes@virginia.edu}

\begin{abstract}
We study singular radially symmetric solution to the Lin-Ni-Takagi equation for a supercritical power non-linearity in dimension $N\geq 3$. It is shown that for any ball and any $k \geq 0$, there is a singular solution 
that satisfies Neumann boundary condition and oscillates at least $k$ times around the constant 
equilibrium. Moreover, we show that the Morse index of the singular solution is finite 
or infinite if the exponent is respectively larger or smaller than the 
Joseph--Lundgren exponent . \\
\smallskip

\noindent \textbf{Keywords.} Bifurcation branches, radial Lin-Ni-Takagi equation, oscillations, singular solutions, Morse index. 
\end{abstract}

\maketitle

\section{Introduction}
In the present paper we study singular solutions of the problem
\begin{equation}\label{LNT}
\left\{
\begin{aligned}
-\Delta v+v &= v^p & \qquad &\text{in } B_R \setminus\{0\} \,, \\
\quad v &> 0 & &\text{in }B_R \setminus\{0\} \,, \\
\partial_\nu v &= 0 &\quad &\text{on } \partial B_R\,,
\end{aligned}
\right. 
\end{equation}
where $B_R \subset \R^N$, $N \geq 3$  is a ball of radius $R > 0$ centred at the origin. We show that for sufficiently large $p$ (super-critical), \eqref{LNT} possesses
a radial solution with many oscillations. The importance of singular solutions 
stems from the fact that they are asymptotes to the bifurcation branches of
regular solutions. Therefore, we investigate the Morse index of singular solutions, which indicates the oscillatory or non-oscillatory nature of the bifurcation branches. 

The problem \eqref{LNT}  arises as a particular case of the stationary 
 Keller-Segel system which is a reaction-diffusion system modelling chemotaxis 
 -- oriented motion of cells toward higher or lower 
 concentrations of chemicals. One can also derive \eqref{LNT} from the 
activator-inhibitor system proposed by Gierer and Meinhardt \cite{Gierer-Meinardt}
under the assumption
that one chemical diffuses much faster than the other one. Gierer-Meinhardt system was extensively studied during last decades (see for example \cite{MR2497911} and references therein), since it is one 
of the first and simplest examples of the diffusion driven instability.

Equation \eqref{LNT} with sub-critical exponent on smooth domains has been intensively studied in the last decades. More precisely, consider
\begin{equation}\label{LNTgen}
\left\{
\begin{aligned}
-\Delta v+\tilde{\lambda} v &= v^p & \qquad &\text{in } \Omega \,, \\
\quad v &> 0 & &\text{in }\Omega \,, \\
\partial_\nu v &= 0 &\quad &\text{on } \partial \Omega\,,
\end{aligned}
\right. 
\end{equation}
where $\Omega\subset \R^N$, $N\geq 3$, is an open smooth domain, $\tilde{\lambda} >0$ and $1\leq p\leq 2^\ast -1=\frac{N+2}{N-2}$. 
In a series of seminal works \cites{LNT,NT1,NT2}, Lin, Ni, and Takagi proved the existence of families of solutions concentrating around one or more points. 
Specifically, they showed that, when  $1<p<2^\ast -1$, then the least energy solution $u_{\tilde{\lambda}}$ to \eqref{LNTgen} satisfies
$$u_{\tilde{\lambda}} \to \tilde{\lambda}^{\frac{1}{p-1}} U (\sqrt{\tilde{\lambda}} (x-x_{\tilde{\lambda}})), 
\qquad \textrm{as } \tilde{\lambda} \to \infty
$$
where $x_{\tilde{\lambda}} \in \partial\Omega$ converges to the boundary point with the maximal mean curvature and the asymptotic profile $U$ is the unique positive radial solution to
$$-\Delta U+U=U^p,\ \lim_{|x|\rightarrow \infty}U(y)=0, \text{ in } \R^N.$$

Observe that the parameter $\tilde{\lambda}$ is related to the size of the domain. Indeed, let $v$ we be for example a solution of
 \eqref{LNTgen} with $\Omega = B_{R}$. Then,  by setting $u(x)=R^{\frac{2}{p-1}} v(Rx)$, we see that $u$ satisfies 
 \begin{equation}\label{lnt-rad}
 -\Delta u +R^2 u=u^p,\textrm{ in}\ B_1.
  \end{equation}
We refer to \cites{dancershusen,MR3606457,MR2812575,MR2395219} and the references therein for construction and analysis  of families of solutions concentrating on multiple points located either in the interior of $\Omega$ and/or at the boundary. 

The position of spikes is more restricted if $p$ is critical, that is, if $p=2^\ast -1$. 
Then, it is possible to show the concentration/bubbling phenomena when $\lambda \rightarrow \infty$ with asymptotic profile being the standard bubble, that is,  
the unique 
(up to scaling and translations) solution to
$$-\Delta U=U^{2^\ast -1},\ \text{in } \R^N.$$
When $p$ is critical and $N=3$ or $N\geq 7$, it has been proved that there is no solution bubbling at an interior point of the domain \cites{MR2334996,MR1707889}. Moreover, in arbitrarily dimension, interior bubbling solutions can only exist if they are bubbling also at the boundary \cite{MR1968337}.
We refer to \cite{MR2114411} for construction of families of bubbling solutions when $p$ approaches from below or above the critical one. We also very briefly point out that families of solutions concentrating on higher dimensional object have been obtained see for instance \cites{MR3801293,MR3262455} and the references therein.

In the supercritical case  $p>2^\ast -1$, there is another significant exponent found by Joseph and Lundgren \cite{josephlundgren}
$$p_{JL}=\begin{cases} 1+\frac{4}{N-4-2\sqrt{N-1}},& N\geq 11,\\ \infty,& 3\leq N\leq 10. \end{cases} \,, $$
which is connected to the number of intersections of radial solutions, and therefore to the stability with respect to compactly supported
perturbations.  
First bifurcation results in supercritical range were obtained for radial solutions solving
 \eqref{lnt-rad} with Dirichlet boundary conditions, or 
more generally for 
\begin{equation}
\label{eqjl}
\begin{cases}U_{rr}+\dfrac{N-1}{r}U_r + \lambda g(U)=0,\ 0<r<1,\\
U>0,\ 0<r<1,\\
U'(0) = U(1)=0,\ \end{cases}
\end{equation}
see \cites{buddnorbury,dolbeaultflores,guowei,josephlundgren,merlepeletier,miya2}. If $g(U)=e^U$ or $g(U)=(1+U)^p$,  
Joseph and Lundgren in \cite{josephlundgren} (see also 
\cite[Chapter 2]{MR2779463} for a recent survey) showed that there is a curve of positive solutions to \eqref{eqjl} starting from the trivial solution $U=0$ and $\lambda =0$. 
 Moreover, they proved that this branch oscillates around a fixed value of $\lambda$ when $3\leq N\leq 9$ in case $g(U)=e^U$ respectively $p_S<p<p_{JL}$ when $g(U)=(1+U)^p$
 (see Theorem \ref{bifmiya} below for more precise statement). On the other hand, the branch does not oscillate when $N\geq 10$ respectively $p\geq p_{JL}$.
Let us point out that Gel'fand \cite{MR0153960} was the first one who treated the case $N=3$ for $g(U)=e^U$.

In the literature, there are fewer results for \eqref{lnt-rad} with Neumann conditions. The first reason might be that there are infinitely many branches 
with positive radial solutions, which are harder to analyze. The second reason might be practical, as in the Dirichlet case the radial solutions are 
stable at least in some parameter ranges, whereas in Neumann case, the radial solutions have large Morse index in the space of all (even non-radial) functions. 
Nevertheless, the bifurcation results were obtained by 
Miyamoto \cite{miya1}, who analyzed bifurcations of radial solution to \eqref{LNTgen} with respect to the parameter $\tilde{\lambda}$. 
To keep make the notation compatible with  \cite{miya1}, note that after scaling $v$ we can 
rewrite \eqref{LNTgen} for $\Omega = B_R$ as
\begin{equation}\label{LNTmiya}
\left\{
\begin{aligned}
-\Delta v &=\lambda (v^p -v) & \qquad &\text{in } B_R \,, \\
\quad v &> 0 & &\text{in } B_R \,, \\
\partial_\nu v &= 0 &\quad &\text{on } \partial B_R \,.
\end{aligned}
\right. 
\end{equation}

\begin{thm}[\cite{miya1}]
\label{bifmiya}
Suppose that $p>2^\ast -1$. Let $\mathcal{S}$ denote the set of the regular, radial solutions of \eqref{LNTmiya}. Then
$$\mathcal{S}=\mathcal{C}_0 \cup \bigcup_{n=1}^\infty \mathcal{C}_n \,,$$
where $\mathcal{C}_0 = \{\lambda, 1\}$. Moreover, since solutions are radial and $v'(0) = 0$, 
each $\mathcal{C}_n$ can be parametrized by $\gamma = v(0) \in (0, \infty)$,  
hence $\mathcal{C}_n = \{(\lambda_n (\gamma), v( \cdot ,\gamma, \lambda_n (\gamma)))\}$. Furthermore, $\gamma \mapsto \lambda_n (\gamma ) \in C^1 (0,\infty)$ and the following assertions hold :
\begin{itemize}
\item[(i)] For each $n\geq 1$, $\lambda_n (1)= \bar{\lambda}_n$, where $\bar{\lambda}_n=\frac{\mu_n}{p-1}$ and $\mu_n$ is the $n$-th eigenvalue of the Laplacian with Neumann boundary condition.
\item[(ii)] For each $n\geq 1$, there exists $\lambda_n^\ast>0$ such that $\lambda_n (\gamma )\rightarrow \lambda_n^\ast$ as $\gamma \rightarrow \infty$.
\item[(iii)] If $p<p_{JL} $, then for each $n\geq 1$, $\lambda_n (\gamma )$ oscillates around $\lambda_n^\ast$ infinitely many times as $\gamma \rightarrow \infty$.
More precisely, there exists a sequence $(\gamma_m^n)_m$ with $\gamma_m^n \to \infty$ as $m \to \infty$ such that $\lambda_n (\gamma_m^n) = \lambda^\ast_n$.
\item[(iv)] For each $n\geq 1$, $\lambda_n (\gamma )\rightarrow \infty$ as $\gamma \rightarrow 0^+$.
\item[(v)] For each $\gamma \in (0, \infty)$, $\lambda_1 (\gamma ) < \lambda_2 (\gamma )< \ldots$.
\end{itemize}
\end{thm}

As mentioned above, the parameter $\lambda$ is connected to the size of the domain, which is in many models fixed (see Keller-Segel system and chemotaxis). However, 
other constants such as diffusivity can change, and these, after scaling, are related to $p$. Hence, 
instead of changing the domain,  that is, the parameter $\lambda$, we fix the domain and vary $p$.
We recall bifurcation results in such case from \cite{BonheureGrumiauTroestler2015}. Here and below 
$p_i^{rad}$ denotes the $i$-th radial eigenvalue of the operator $-\Delta+Id$ in
the ball $B_R := \{x \in \mathbb{R}^N : |x| < R\}$ with Neumann boundary conditions.

\begin{thm}[\cite{BonheureGrumiauTroestler2015}]\label{thm:bifurcation}
For every $i\geq 2$, the trivial branch $(p, 1)$ of problem \eqref{LNT} 
has a  bifurcation point at $(p_i^{rad},1)$. 
If $\Bcal_i \subset \R^2$, parametrized by $(p, u(0))$, is the continuum that branches out of $(p_i^{rad},1)$, then the following holds:
\begin{itemize}
\item[(i)] The branches $\Bcal_i$ are unbounded and do not intersect. Furthermore, near $(p_i^{rad},1)$, $\Bcal_i$ is a $C^1$-curve.
\item[(ii)] If $(p, A) \in \Bcal_i$, then the corresponding solution $u_p$ satisfies $u_p > 0$ in $B_R$.
\item[(iii)] Each branch consists of two connected components $\Bcal_i^- := \Bcal_i \cap \{(p, A) : A < 1\}$ and 
 $\Bcal_i^+ := \Bcal_i \cap \{(p, A) : A > 1\}$.
\item[(iv)] If $(p, A) \in \Bcal_i$ then the corresponding $u_p-1$ has exactly $i-1$ zeros, $u_p'$ has exactly $i$ zeros (including ones on the boundary and at the origin).
\item[(v)] The functions satisfying $u_p (0) < 1$ are uniformly bounded in the $C^1$-norm.
\end{itemize}
\end{thm}

Previously, by different techniques the lower branches $\Bcal_i$ were presumably constructed in \cite{BonheureGrossiNorisTerracini2015} and the first upper branch $\Bcal_2$ by \cite{reywei} when $N=3$.

The goal of this paper is to establish oscillatory results for upper branches as in Theorem \eqref{bifmiya} or as in the Dirichlet case. 
In the following, we will only be concerned with upper branches for $p>2^\ast -1$ and their asymptotics when $p$ gets large. 
Of course, in the finite range the branches can have only finitely many turns, and therefore large $p$ behaviour determines on 
oscillations or non-oscillation of branches. We focus on  singular solutions, that are limit profiles of bifurcation branches as proved by Miyamoto in the following theorem.

\begin{thm}[ \cite{miya1}]
\label{LNTthmmiya}
Let $N\geq 3$ and $p>2^\ast -1$. There is a unique solution $U^\ast_p:=U^\ast$ to 
\begin{equation}
\label{LNTwhole}
\begin{cases}
- u^{\prime \prime} - \dfrac{N-1}{r}u^\prime + u = u^{p},& \text{in}\ \R^+,\\
\lim_{r\to 0^+} r^\theta u(r)=A_{p,N},\\
u>0,& \text{in} \ \R^+,
\end{cases}
\end{equation}
where
$$\theta = \frac{2}{p - 1},\ and\  A_{p, N} = [\theta(N - 2 - \theta)]^{\frac{1}{p-1}}.$$
Moreover, $U^\ast$ attains infinitely many times the value $1$. Furthermore, if there are sequences $(\gamma_n)_n$ and $(p_n)_n$ with $\gamma_n \rightarrow \infty$ and $p_n \rightarrow p_\infty >2^\ast -1$, then $u_{\gamma_n, p_n}\rightarrow U^\ast_{p_\infty} $ in $C_{loc}^0 (0,\infty)$, where $u_{\gamma ,p}$ is the solution to
\begin{equation}
\label{LNTeqintfin}
\begin{cases}
- u^{\prime \prime}-\dfrac{N-1}{r}u^\prime + u = u^{p},& \text{in} \ \R,\\
u(0)=\gamma,\ u^\prime (0)=0.
\end{cases}
\end{equation}
\end{thm}

Since $U^\ast$ attains infinitely many times the value $1$, there exists an increasing sequence
 $(R^i_{p})_i$ such that $(U^\ast_{p})^\prime (R_{p}^i)=0$, and therefore $U^\ast_{p}$ is a solution of \eqref{LNT} with $R$ replaced by $R_{p}^i$. However, if the size of the domain
 is fixed, then the existence of singular solution does not follow from Theorem \ref{LNTthmmiya}, unless one is willing to change the equation (or more precisely $\lambda$) by
 scaling as in \eqref{LNTmiya}. 
 
In our main result, we show that, for a fixed radius $R$ and any large integer $i > 1$, we can find $p>2^\ast -1$ such that $R^i_p
= R$. In other words, for any $R$ fixed, we are able to construct a singular solution to \eqref{LNT} having a prescribed number of intersections with $1$ (and therefore a prescribed number of critical points). Since by Theorem \ref{thm:bifurcation} all solutions on $\mathcal{B}^+_i$ have exactly $i$ critical points, we believe that the limit point of $\mathcal{B}^+_i$ is 
exactly the constructed singular solution with $i$ critical points. 

Our theorem also complements the results proved by Lin and Ni \cite{MR974610} that, asserts that for any fixed $p>2^\ast -1$, there exists $R^\ast$ depending on $p$ and $N$ such that, for all $R< R^\ast$, equation \eqref{LNT} only admits constant solutions.  

\begin{thm}
\label{LNTmainthm}
Let $N\geq 3$ and $R>0$. Fix $\tilde{p}>2^\ast -1$ and let $U_{\tilde{p}}^\ast$ be the solution to \eqref{LNTwhole}. Let $i^\ast$ be the smallest integer such that 
$R_{\tilde{p}}^{i^\ast}>R$. Then, for any $i\geq i^\ast$, there exists $p^i>2^\ast -1$ such that
$$
R^i_{p^i}=R.
$$
In particular, for any $i\geq i^*$, there exists $p^i>2^\ast -1$ such that equation \eqref{LNT}  admits a singular radial solution $U$ satisfying
$$
\sharp \{r\in [0,R]| U(r)=1 \}=i.
$$
\end{thm}

We remark that an analogous result with $u^p$ replaced by $\lambda e^u$ (with $\lambda$ as a bifurcation parameter) been obtained by the authors and Bonheure in \cite{BCF}. 

Next, we investigate the asymptotic behavior of the branch $\mathcal{B}_i^+$. The following theorem proved in \cite{miya1} gives a strong indication that for each $i \geq 1$, the branch $\mathcal{B}_i^+$ oscillates around $p^i$ (see Theorem \ref{LNTmainthm})
when $2^\ast -1 <p^i < p_{JL}$. Fix $p>2^\ast -1$ and $\gamma_0$. We denote by $(r_{p,\gamma}^i)_i$, the increasing sequence of positive real numbers satisfying $u_{\gamma,p}^\prime (r_{p,\gamma}^i )=0$, where $u_{\gamma, p}$ is the unique solution to \eqref{LNTeqintfin}.

\begin{thm}\cite[Theorem $6.1$]{miya1}\label{th:mia}
Let $R>0$, $N \geq 11$, $i\geq i^\ast$, and $2^\ast -1 <p^i < p_{JL}$. Then, there exist a sequence of initial data $(\gamma_n)_n$ and a sequence of positive integer $(j_n)_n$ such that $\gamma_n \rightarrow \infty$ and $r_{p^i ,\gamma_n}^{j_n}=R$. 
\end{thm}

Note that since $j_n$ in general depends on $n$, one cannot conclude that 
the points $(p^i, \gamma_n)$ lie on $\mathcal{B}_i$. Also, without additional
information one cannot combine Theorem \ref{th:mia} and Theorem \ref{LNTthmmiya} to prove Theorem \ref{LNTmainthm} by limiting procedure. 
We remark that the oscillations and convergence of $\mathcal{B}_i$ was proved
by authors and Bonheure in \cite{BCF} for \eqref{LNT} with $v^p$ replaced by 
$\lambda e^v$. The proof in the present case is more involved and will be published separately. 

A strong indication that branches oscillate when $p^i < p_{JL}$ and do not oscillate when $p^i>p_{JL}$ is provided by the radial Morse index of our singular solution.
 Recall that the Morse index of $v$ satisfying \eqref{LNT}, denoted by $m(v)$, in the space of radial functions is the number of negative eigenvalues $\alpha$ counted with multiplicity of the following eigenvalue problem
\begin{equation}
\left\{
\begin{aligned}
-\Delta \phi+\phi - p u^{p-1} \phi &= \alpha \phi & \qquad &\text{in } B_R \setminus\{0\} \,, \\
\partial_\nu \phi &= 0 &\quad &\text{on } \partial B_R\,, \\
 \text{   $\phi$ is radially symmetric.}
\end{aligned}
\right. 
\end{equation}
Note that each turn of the bifurcation branch increases the Morse index of solutions, 
thus finite or infinite Morse index of the limit (singular solution) suggest respectively
non-oscillatory or oscillatory behaviour.

\begin{prop}\label{LNTmorse}
Let $U^\ast_{p^i}$ be a solution to \eqref{LNTwhole}, where $p^i$ is as in Theorem \ref{LNTmainthm}. Then $m(U^\ast_{p^i} )<\infty$ if $p^i>p_{JL}$, while $m(U^\ast_{p^i} )=\infty$ if $2^\ast -1 <p^i < p_{JL}$.
\end{prop}

Finally, we briefly sketch main ideas of the proofs. To prove Theorem \ref{LNTmainthm}, we follow the  general framework used in
 \cite{BCF}. Specifically, Theorem \ref{LNTmainthm} is a consequence of continuity of the function $p\rightarrow R^i_p$ 
for all $i\in \N$ and 
\begin{equation}
\label{introe1}
R_i^p \rightarrow 0^+ \text{ as } p\rightarrow \infty \qquad \textrm{for all }
i\in \N. 
\end{equation}
To establish of \eqref{introe1}, as in 
 \cite{BCF}, we obtain very precise estimates of $U^\ast_p$ in a neighbourhood of  the origin. It is crucial to control the size of the neighbourhood with respect
 to parameter $p$. The proof is rather technical and requires very detailed
 information about solutions. Unlike in \cite{BCF}, our estimates cease to 
 hold before the first intersection point with $1$ that we denote $r_p$. At least 
 heuristically $r_p \approx \frac{1}{\sqrt{p}}$ (in fact the upper bound can 
 be made rigorous). Although we are not able to control the solution till $r_p$
 we obtain estimates on the interval of comparable length  $[0,\frac{\tilde{c}}{p} ]$, where $\tilde{c}$ is  sufficiently small constant.  The key ingredient is the negativity  of the higher order correction of $U^\ast_p$. Note that such estimate 
 would not suffice in \cite{BCF}, however since our constant equilibrium (equal to 1)
 is independent of $p$, we could proceed. 
 
Consequently, we prove that $(U^\ast_p)^\prime ( ,\frac{\tilde{c}}{p})$ converges to $0$ when $p\to \infty $. Using the decay of an energy functional, we show that $U^\ast_p (r)$ stays very close to $1$ for any $r\geq \frac{\tilde{c}}{p}$ and  we are conclude by using the Sturm-Piccone theorem. 

The continuity of the function $p \rightarrow R^i_p$ relies heavily on the uniqueness of $U^\ast_p$ and again the precise estimates at the origin on a controlled interval. 
 
Proposition \ref{LNTmorse} containing the estimates on the Morse index of $U^\ast_p$ relies on the the asymptotic behaviour of $U^\ast_p$ when $r\to 0$ and the Hardy's inequality.

\section{Proof of Theorem \ref{LNTmainthm}.}
In this section, we prove Theorem \ref{LNTmainthm}. It will be a consequence of the continuity of the function $p\rightarrow R_p^i$, for all $i\in \N$ and the fact that
\begin{equation}\label{s2e1} R_p^i \rightarrow 0^+,\ \text{ as } p\rightarrow \infty .\end{equation}
First, we prove that \eqref{s2e1} holds true. In all the following, we denote by $U^\ast :=U^\ast_p$ the singular solution of \eqref{LNTwhole}. Before proceeding, let us give several definitions and recall some facts. We begin by introducing a change of variables which was already used in \cite{miya1} to prove the existence of a singular solution.

Define
\begin{equation}
\label{LNTdefy}
\eta (\zeta ) = A_{p, N}^{-1} r^{\theta} U^\ast (r )-1, \qquad  -\zeta = m^{-1} \ln r \, ,
\end{equation}
where $A_{p,N}$ and $\theta$ are defined in Theorem \ref{LNTthmmiya} and 
$$m = [\theta(N - 2 - \theta)]^{- \frac{1}{2}} .$$
It is easy to check that $\eta$ satisfies
\begin{equation}
\label{LNTeqy}
\begin{cases}
\eta '' - \alpha  \eta ' +(p-1) \eta = -(1+\eta)^p + 1+ p\eta + m^2 e^{-2m\zeta}(1+\eta) ,\ \text{on}\ \R ,\\
\lim_{\zeta\rightarrow \infty} \eta (\zeta)=0, 
\end{cases}
\end{equation}
where
$$\alpha = m(N - 2 - 2\theta).$$
Next, we set $\tilde{\eta}=\eta -f$, where $f(\zeta)=D_p e^{-2m \zeta}$ and $D_p=\frac{m^2}{4m^2+2\alpha m +(p-1)}$. Then, a straightforward computation shows that
\begin{equation}\label{eqtildeeta}
\tilde{\eta}^{\prime \prime} -\alpha \tilde{\eta}' +(p-1)\tilde{\eta}= m^2 e^{-2m\zeta} \eta +\phi (\eta ) =: \tilde{g}, \end{equation}
where \begin{equation}\label{defph}\phi (\eta )= - ( (1+\eta)^p -1 - p \eta ).
\end{equation}

 We will also make intensive use of the following representation formula :
\begin{equation}
\label{repform}
\tilde{\eta}=\int_\zeta^\infty G_N (\sigma -\xi) \tilde{g} (\sigma ) d\sigma,
\end{equation}
where
$$
G_N(x)=
\begin{cases} 
\frac{e^{-\frac{\alpha}{2}x }}{\beta} \sin (\beta x), & \textrm{if } p-1> (\alpha /2)^2 \\  
\frac{e^{-\frac{\alpha}{2}x }}{\beta} \sinh (\beta x), &  \textrm{if } p-1< (\alpha /2)^2 \\ 
 e^{-\frac{\alpha}{2} x}x,  &  \textrm{if } p-1= (\alpha /2)^2
 \end{cases},
 \quad
 \textrm{for } 
   x\geq 0,
    \qquad\qquad
     G_N (x)=0,  \textrm{ if } x< 0.
 $$
and $\beta = \sqrt{|p-1 - (\alpha /2)^2|}$. Using that $\lim_{p\rightarrow \infty}\frac{1}{p-1} (\frac{\alpha}{2})^2 = \frac{N-2}{8}$ and 
$\frac{1}{p - 1}(\frac{\alpha}{2})^2 = \frac{(8 - 2\theta)^2}{8(8 - \theta)}$ when $N=10$, we deduce that, for $p$ large enough, $p-1> (\frac{\alpha}{2})^2$ if $N\leq 10$ and $p-1< (\frac{\alpha}{2})^2$ if $N> 10$.

We also define $w(r)=r^{\frac{N-1}{2}}(U^\ast (r)-1)$. By standard manipulations, one has
\begin{equation}
\label{eqw}
w^{\prime \prime}+ \left(\frac{(U^\ast)^p - U^\ast}{U^\ast-1} -\frac{(N-1)(N-3)}{4r^2} \right) w=0.
\end{equation}
The following asymptotics when $p\rightarrow +\infty$ of parameters are useful below 
\begin{gather}
\label{LNTlem1asy1}
\lim_{p\rightarrow \infty}\frac{\beta}{\sqrt{p}}=   \sqrt{\left|1- \frac{N-2}{8}\right|} \quad \textrm{if } N\neq 10, \qquad  
\beta = \sqrt{\frac{3(p - 1) - 1}{4(p - 1) - 1}}\quad \textrm{if }  N=10, \\
\lim_{p\rightarrow \infty} p\theta =2,\ \lim_{p\rightarrow \infty} A_{p,N} =1,\ \lim_{p\rightarrow \infty} \frac{\alpha}{\sqrt{p}} =\sqrt{\frac{N-2}{2}},\ \lim_{p\rightarrow \infty}\frac{ m}{\sqrt{p}}=\frac{1}{\sqrt{2(N-2)}},\ \lim_{p\rightarrow \infty} D_p= \frac{1}{4(N-1)}.
\end{gather}
If precise constants are not necessary, we use the notation $A\approx p^b$ for some real number $b$ if there exist two positive constants $c_1$ and $c_2$ such that, $c_1 \leq \frac{A}{p^b} \leq c_2$, for $p$ large. We also use the notation $A_p=O(p^{-b})$ if there exists a constant $C$ not depending on $p$, 
such that $|A_p| \leq C p^{-b} $ for any large $p$. First, we provide an upper bound (for $p$ large) for the first intersection of the singular solution with the value $1$.
Let us prove an auxiliary lemma first. 

\begin{lem}\label{l:boct}
There exists $0<\tilde{c}<1$ such that for any sufficiently large $p$,
\begin{equation}
\label{deftildec}
P_N <
\begin{cases}
\dfrac{1}{2}  \dfrac{1-  e^{-\frac{(\alpha + 8m) \pi}{2\beta}}}{1+ e^{-\frac{(\alpha + 8m) \pi}{2\beta}}} & \textrm{if } N < 10\,, \\
\dfrac{1}{2}  & \textrm{if } N \geq 10 \,,
\end{cases}
\end{equation}
where 
\begin{equation}
\label{defPN}
P_N := 
\frac{\left|\phi\left(f(\tilde{\zeta}_p ) \right) \right|}{ f(\tilde{\zeta}_p)}  \times 
\begin{cases} 
\dfrac{4}{(\alpha + 8m)^2 + 4\beta^2} (1 + e^{-\frac{(\alpha + 8m) \pi}{2\beta}})& \textrm{if } N < 10,\\ 
 \frac{1}{2\beta (\frac{\alpha}{2} +4m -\beta)} & \textrm{if } N \geq 10,
 \end{cases}
\end{equation}
and 
\begin{equation}
\label{relrzeta}
r_p = e^{- m_p \zeta_p}, \qquad \frac{\tilde{c}}{\sqrt{p}} = e^{- m_p \tilde{\zeta}_p} = \sqrt{\frac{f(\tilde{\zeta}_p)}{D_p}} \,.
\end{equation}
\end{lem}

\begin{proof}
First, notice that
\begin{equation}\label{tab}
(\alpha +8 m)^2 + 4 \beta^2 \approx 
p  \quad \textrm{for } N < 10 \,, \qquad 
2\beta \left(\frac{\alpha}{2}+4m - \beta\right) \approx  p \quad \textrm{for } N \geq 10 ,
\end{equation}
and
$$
\frac{\alpha +8m}{2\beta}\approx 1  \quad \textrm{for } N < 10 \,.
$$
In addition, since $f(\tilde{\zeta}_p) = D_p \tilde{c}^2/p$ and $D_p \approx 1$, we can choose $\tilde{c}$ sufficiently small such that $k := D_p \tilde{c}^2 \leq 1$. 
Then using that $p \mapsto (1 + k/p)^p$ increases to $e^k$, we obtain 
\begin{equation} \label{phes}
\left|\phi\left(\frac{k}{p}\right)\right| =  \left| \left(1 + \frac{k}{p}\right)^p  - k - 1\right| \leq |e^{k} - k - 1|  \leq c_N k^2 \,. 
\end{equation}
Consequently,
\begin{align}
\frac{\left|\phi\left(f(\tilde{\zeta}_p ) \right) \right|}{ f(\tilde{\zeta}_p)} &= \frac{p}{\tilde{c}^2 D_p} |\phi\left(D_p \tilde{c}^2/p \right) |   
 \leq c_N p D_p\tilde{c}^2 
\end{align}
and from \eqref{tab} follows
\begin{equation}\label{dfct}
P_N \leq  p D_p\tilde{c}^2 \frac{C_N }{p}  = C_N  D_p\tilde{c}^2 \,.
\end{equation}
Hence,  \eqref{deftildec} is satisfied for some sufficiently small $\tilde{c}$ independent of $p$ as desired. 
\end{proof}

In the rest of the proof, we fix $\tilde{c}$ such that Lemma \ref{l:boct} holds. 
 Fix any $\varepsilon_0 > 0$ and set
\begin{equation}\label{defzeta2ast}
\zeta_1^* := \inf\{ \zeta\geq \tilde{\zeta}_p : |\tilde{\eta}(z)| \leq (1+ \varepsilon_0) P_N f(z) \textrm{ for any } z \geq \zeta\},
\end{equation} 
To simplify notation, we set $P_{N, \varepsilon_0} := P_N(1 + \varepsilon_0)$. 
First, we show that $\zeta^\ast_1$ is well-defined.

\begin{lem}
For any $p>2^\ast -1$ and any $\varepsilon_0 >0$, we have $\zeta^\ast_1 <\infty$. 
\end{lem}

\begin{proof}
Fix any $\varepsilon > 0$. 
First, notice that  
 $$
  \|G_N\|_{L^1}\leq C_{N,p}=
  \begin{cases} 
  \frac{2}{\alpha \beta} ,\ \textrm{if} \ N < 10,\\
 \frac{2}{\alpha^2}  
  \frac{1}{\beta (\alpha - 2\beta)},\ \textrm{if} \ N>10 ,
  \end{cases}
  $$
  and $C_{N,p} \approx p^{-1}$ if $N \neq 10$ and $C_{N,p} \approx  p^{-1/2}$ if $N = 10$.
Using the representation formula \eqref{repform} and Young's inequality for convolutions, we obtain 
\begin{equation}\label{imca}
\int_\zeta^\infty |\tilde{\eta} (\sigma )|d\sigma \leq C_{N,p} \int^\infty_{\zeta}   |\tilde{g}(\sigma)|  d\sigma \, .
\end{equation}
Since the function $x \mapsto |\phi(x)|/x$ is increasing, then \eqref{phes} implies
\begin{equation}\label{LNTiqe}
\frac{|\phi(\eta)|}{\eta} \leq 
\frac{|\phi(\varepsilon/p)|}{\varepsilon/p} \leq 
c_N p \varepsilon
 \qquad \textrm{for all } 0 < \eta  \leq \frac{\varepsilon}{p} \,. 
\end{equation}
On the other hand, since $\eta(\zeta ) \to 0$ as $\zeta \to \infty$ (see \eqref{LNTeqy}), we deduce that there exists $\zeta_0>0$ depending on $p$ and $\eta$ such that 
$|\eta (\zeta )| \leq \varepsilon/p$  for any $\zeta \geq \zeta_0$, and consequently by the definition of $\tilde{\eta}$
\begin{equation}
\label{LNTdefzeta0}
 |\phi(\eta (\zeta )| \leq 
c_N p \varepsilon |\eta(\zeta)|\leq   
 c_N p \varepsilon ( | \tilde{\eta} (\zeta )| + |D_p e^{-2m\sigma} |).
\end{equation}
Recalling the definition of $\tilde g$ (see \eqref{eqtildeeta}), one has for $\sigma \geq \zeta_0$, 
\begin{equation*}
|\tilde{g} (\sigma)| \leq (c_N\varepsilon p + m^2 e^{-2m\sigma}) (|\tilde{\eta}(\sigma)| + D_pe^{-2m\sigma}) \,. 
\end{equation*}
Since $m \approx \sqrt{p}$ and $D_p e^{-2m \zeta} = f(\zeta) \leq \varepsilon/p$ for   $\zeta \geq  \zeta_0$, then 
$m^2 e^{-2m\sigma} \leq c_N \varepsilon$,  and therefore
\begin{equation}
\label{LNTlem1e2}
|\tilde{g}(\sigma)| \leq  2c_N \varepsilon p (|\tilde{\eta}(\sigma)| + D_pe^{-2m\sigma}) \,.
\end{equation}
Substituting this estimate into \eqref{imca}, we obtain, for $\zeta \geq  \zeta_0$ and any sufficiently large $p \geq c_0$, 
\begin{equation}
\left(1 - 2 c_N \varepsilon p C_{N,p} \right) \int_{\xi}^\infty |\tilde{\eta} (\sigma)|\, d\sigma \leq \frac{c_N \varepsilon p D_p}{m} C_{N,p} e^{-2m \zeta} .
\end{equation}
We decrease $\varepsilon_0$ if necessary to have
\begin{equation}
\label{LNTlem1asy4}
\varepsilon_0 <\frac{1 }{4c_N pC_{N,p}}, \quad \textrm{and therefore } \left(1 -2c_N \varepsilon p C_{N,p} \right) \geq \frac{1}{2}, \qquad  \qquad \textrm{for any } \varepsilon \in (0, \varepsilon_0) \,.
\end{equation}
Hence, 
\begin{equation}\label{req}
 \int_{\xi}^\infty |\tilde{\eta} (\sigma)|\, d\sigma \leq \frac{c_N D_p}{m} e^{-2m \zeta}.
\end{equation}
Next, we use  $\|G_N\|_{L^\infty} \leq C_N$ combined with 
\eqref{repform}, \eqref{LNTlem1e2}, \eqref{req} and Young convolution inequality to get that 
\begin{equation}
|\eta(\zeta)| \leq C_N \int_{\zeta}^\infty |\tilde{g}(\sigma)|\, d \sigma \leq
 C_N \varepsilon p 
  \int_\zeta^\infty  (D_p e^{- 2m\sigma} + |\tilde{\eta}(\sigma)|) \ d\sigma
\leq  C_N \varepsilon p \left( \frac{D_p}{m}  + \frac{c_N D_p}{m} \right) e^{-2m \zeta} \,.
\end{equation}
Then, the definition of $\varepsilon$ yields 
\begin{equation}
|\eta(\zeta)| \leq \frac{c_ND_p}{m C_{N, p}}  e^{-2m \zeta} \qquad \textrm{for any } \zeta \geq \zeta_0 \,.
\end{equation}
Since $D_p \approx 1$, $m \approx \sqrt{p}$, and $C_{N, p} \geq 1$ for large $p$, we obtain the desired conclusion. 
\end{proof}

\begin{lem}
\label{lemetaneg}
 For any small $\varepsilon_0 > 0$, there exists $p_0 > 0$ such that, for each $p \geq p_0$, we have $\eta \leq 0$ on $[\zeta_1^*, \infty)$ where $\zeta_1^\ast$ is defined in \eqref{defzeta2ast}.
\end{lem}
\begin{proof}
We first assume that $N\leq 10$. We rewrite \eqref{repform} as
\begin{equation}
\label{KSdefetainto0}
\tilde{\eta} (\zeta )= \int^\infty_{\zeta} G_N(\sigma - \zeta) \tilde{g}(\sigma )  d\sigma =: 
 \int^\infty_{\zeta} F(\zeta, \sigma)  d\sigma 
\end{equation}
and
\begin{align}\label{frl}
\int^\infty_{\zeta} F(\zeta, \sigma)  d\sigma &= \sum_{k = 0}^\infty \int_{\zeta+\frac{2k\pi}{\beta}}^{\zeta+\frac{(2k+1)\pi}{\beta}} F(\zeta, \sigma) d\sigma + \int_{\zeta +\frac{(2k+1)\pi}{\beta}}^{\zeta + \frac{(2k+2)\pi}{\beta}} F(\zeta, \sigma) d\sigma   
\nonumber\\
&=  \sum_{k = 0}^\infty \int_{\zeta+\frac{2k\pi}{\beta}}^{\zeta+\frac{(2k+1)\pi}{\beta}} F(\zeta, \sigma) + F\left(\zeta, \sigma + \frac{\pi}{\beta}\right) d\sigma  \,,
\end{align}
where
\begin{align}
F(\zeta, \sigma) + F\left(\zeta, \sigma + \frac{\pi}{\beta}\right) = G_N(\sigma - \zeta) \left(\tilde{g}(\sigma) -   e^{-\frac{\alpha\pi}{2\beta}} \tilde{g}\left(\sigma + \frac{\pi}{\beta}\right)\right) \,.
\end{align}
Recall, for any $\sigma \geq \zeta_1^*$ we have 
$|\tilde{\eta}(\sigma)| \leq P_{N,\varepsilon_0} f(\sigma)$, with $1 >  P_{N,\varepsilon_0}$ for any sufficiently small $\varepsilon_0 > 0$. Since $\phi$ is decreasing on $(0, \infty)$ and $f \pm \tilde{\eta} \geq 0$ on $[\zeta_1^*, \infty)$, one has 
\begin{multline}
\phi((f + \tilde{\eta})(\sigma)) - e^{-\frac{\alpha\pi}{2\beta}} \phi \left( (f + \tilde{\eta}) \left( \sigma+ \frac{\pi}{\beta} \right) \right)
 \leq \phi((f - |\tilde{\eta}|)(\sigma)) - e^{-\frac{\alpha\pi}{2\beta}} \phi \left( (f + |\tilde{\eta}|) \left( \sigma+ \frac{\pi}{\beta} \right) \right)  
\\
\begin{aligned}
&\leq \phi((1 - P_{N,\varepsilon_0}) f(\sigma)) - e^{-\frac{\alpha\pi}{2\beta}} \phi \left((1 + P_{N,\varepsilon_0})f \left( \sigma+ \frac{\pi}{\beta} \right) \right)  
\\
&\leq  \phi((1 - P_{N,\varepsilon_0}) f(\sigma)) - e^{-\frac{\alpha\pi}{2\beta}} \phi \left((1 + P_{N,\varepsilon_0}) e^{-\frac{2\pi m}{\beta}} f \left( \sigma \right) \right).
\end{aligned}
\end{multline}

We claim that  for any sufficiently small $\varepsilon_0, \varepsilon_1 > 0$ and any sufficiently large $m$ (that is large $p$), one has
\begin{equation}\label{acnd}
\phi \left( (1 -P_{N,\varepsilon_0}) z \right) \leq e^{-\frac{\alpha \pi}{2\beta}} \phi \left( (1 + P_{N,\varepsilon_0}) e^{-\frac{2\pi m}{\beta}} z \right) 
- 2 (1+P_{N,\varepsilon_0}) \frac{m^2}{D_p}  z^2  , \qquad \textrm{for any } z \in \left[0, K_M/ p\right] \,.
\end{equation}
Indeed, it is easy to check that both value and the value of the derivatives of both sides in \eqref{acnd} vanish at $z = 0$. Thus, it suffices to verify that 
the second derivative of the right hand side is larger than the second derivative of the left hand side on the interval $ \left[0, K_M/ p\right]$. It is 
equivalent to 
\begin{multline}
p(p-1) (1 -P_{N,\varepsilon_0})^2 \left(1 +  (1 -P_{N,\varepsilon_0}) z \right)^{p-2} \\
\geq 
p(p-1)(1 + P_{N,\varepsilon_0})^2 e^{- \frac{\pi}{2\beta} (\alpha + 8m)}  \left( 1 + (1 + P_{N,\varepsilon_0}) e^{-\frac{2\pi m}{\beta}} z \right)^{p-2} +  4 (1+P_{N,\varepsilon_0}) \frac{m^2}{D_p}\,.
\end{multline}
However, by \eqref{deftildec}, 
$$
P_{N} < \frac{1}{2} \frac{1-  e^{- \frac{\pi}{2\beta} (\alpha + 8m)} }{1+ e^{- \frac{\pi}{2\beta} (\alpha + 8m)}} < \frac{1}{2}
$$
and by \eqref{LNTlem1asy1} 
\begin{align}
\frac{\alpha + 8m}{2\beta} &= \frac{N + 6}{\sqrt{(N - 2)(10 - N)}} + O(1) \geq 2\sqrt{2} + O(1) \qquad \textrm{for } N \in [3, 10) \,, \\
\frac{m}{\beta} &= \frac{2}{\sqrt{(N - 2)(10 - N)}} + O(1) \geq \frac{1}{2} + O(1)
 \qquad \textrm{for } N \in [3, 10) \,.
\end{align}
For $N = 10$, the left hand side diverges to infinity, so the latter estimates are still valid. 
Thus, for any sufficiently small $\varepsilon_0$ and large $p$, we have
\begin{align}
 1 -P_{N,\varepsilon_0} \geq \frac{1}{2} &\geq  \frac{3}{2} e^{-\pi + O(1)} 
 \geq (1 + P_{N,\varepsilon_0}) e^{-\frac{2\pi m}{\beta}} \\
 (1 -P_{N,\varepsilon_0})^2 &\geq \frac{1}{4} \geq \frac{9}{2} e^{-2\sqrt{2} \pi + O(1)}
 \geq 2(1 + P_{N,\varepsilon_0})^2 e^{- \frac{\pi}{2\beta} (\alpha + 8m)} 
\end{align}
So, we obtain, for any small $\varepsilon_0 > 0$, and large $p$
\begin{multline}
 (1 -P_{N,\varepsilon_0})^2 \left(1 +  (1 -P_{N,\varepsilon_0}) z \right)^{p-2} -
(1 + P_{N,\varepsilon_0})^2 e^{- \frac{\pi}{2\beta} (\alpha + 8m)}  \left( 1 + (1 + P_{N,\varepsilon_0}) e^{-\frac{2\pi m}{\beta}} z \right)^{p-2} \\
\geq \frac{1}{2}
(1 -P_{N,\varepsilon_0})^2  \left(1 +  (1 -P_{N,\varepsilon_0}) z \right)^{p-2} \geq  \frac{1}{2}
(1 -P_{N,\varepsilon_0})^2 
\,.
\end{multline}
Since $m^2 \approx p$,  \eqref{acnd} follows for any sufficiently large $p$.

In addition, using that $f$ is decreasing and that $|\tilde{\eta} (\sigma )|\leq P_{N,\varepsilon_0}f(\sigma)$, we have, for $\sigma \geq \zeta_1^* $, 
\begin{align*}
m^2 e^{-2m\sigma} \left( (\tilde{\eta} +f)(\sigma ) - e^{-\frac{\pi}{\beta}(2m+\alpha/2) } (\tilde{\eta} +f) \left(\sigma +\dfrac{\pi}{\beta}\right)  \right) &\leq 
2 (1+P_{N,\varepsilon_0}) m^2 e^{-2m\sigma} f (\sigma)  \\
&= 2 (1+P_{N,\varepsilon_0}) \frac{m^2}{D_p}  f^2(\sigma) .
\end{align*}

Therefore, recalling that $\tilde{g}(\zeta )= \phi (\eta(\zeta ))+m^2 e^{-2m\zeta} \eta(\zeta )$, the previous bound combined with \eqref{acnd} implies
\begin{equation}\label{ginq}
\tilde{g}(\sigma) -   e^{-\frac{\alpha\pi}{2\beta}} \tilde{g}\left(\sigma + \frac{\pi}{\beta}\right) \leq 0.
\end{equation}
Since $G_N \geq 0$ on $(\zeta+\frac{2k\pi }{\beta},\zeta+\frac{2(k+1)\pi }{\beta} )$, we obtain that 
$$F(\zeta, \sigma) + F\left(\zeta, \sigma + \frac{\pi}{\beta}\right) \leq 0.$$
Using \eqref{frl} and \eqref{KSdefetainto0}, this established the proof 
for $N \leq 10$.

Next, assume $N > 10$ and notice that $G_N \geq 0$ in this case.  Also, since $|\tilde{\eta}(\sigma)| \leq P_{N, \varepsilon_0} f(\sigma)$ on $[ \zeta_1^* ,\infty)$ and 
$P_{N, \varepsilon_ 0} < 1$ for any sufficiently small $\varepsilon_0$, we obtain that $\eta = f + \tilde{\eta} \geq 0$ on $[\zeta_1^*, \infty)$. Since $(1+x)^p-1-px\geq \dfrac{p(p-1)}{2}x^2$ for $x \geq 0$, then for any $\zeta \geq \zeta_1^*$ 
\begin{align*}
\tilde{\eta} (\zeta )\leq  \int^\infty_{\zeta} G_N(\sigma - \zeta) \Big(m^2 e^{-2\sigma} \eta(\sigma )- \frac{p(p-1)}{2} \eta^2 (\sigma ) \Big)  d\sigma .
\end{align*}
Also, since $\eta \geq 0$, we have 
\begin{align}
m^2 e^{-2\sigma} \eta(\sigma )- \frac{p(p-1)}{2} \eta^2 (\sigma ) &\leq \eta(\sigma ) \left( m^2\frac{f(\sigma)}{D_p} -  \frac{p(p-1)}{2}(f(\sigma) - |\tilde{\eta}(\sigma)|) 
\right)
\\
&\leq \eta(\sigma ) f(\sigma) \left( \frac{m^2}{D_p} - \frac{p(p-1)}{2} (1 - P_{N, \varepsilon_0}) \right) \leq 0 \,,
\end{align}
where we used $\frac{m^2}{D_p}\approx p$ in the last  inequality. Thus $\tilde{\eta}(\zeta) \leq 0$ for each $\zeta \geq \zeta_1^*$ as desired.

\end{proof}

\begin{lem}
\label{tildeetapetit}
 For any sufficiently small $\varepsilon_0 > 0$, there exists $p_0$ such that, for each $p \geq p_0$, we have 
$\zeta_1^* = \tilde{\zeta}_p$, where $\zeta_1^\ast$ is defined in \eqref{defzeta2ast}. In particular, 
$|\tilde{\eta} (\tilde{\zeta_p} )|< \frac{f (\tilde{\zeta}_p)}{2}$.
\end{lem}
\begin{proof}
In Lemma \ref{lemetaneg}, we proved that $\tilde{\eta} \leq 0$ on $(\zeta_1^*, \infty)$. In order to obtain an estimate on $|\tilde{\eta}|$, we need a lower bound on $\tilde{\eta}$. 

First, let us assume that $N\leq 10$. Since $ G_N(\sigma - \zeta) \leq 0$  on the interval $\left(\zeta+\frac{(2k+1)\pi}{\beta}, 
\zeta+\frac{(2k+2)\pi}{\beta}\right)$,   
\eqref{ginq} and \eqref{frl} yield on such interval
$$
F(\zeta, \sigma) + F\left(\zeta, \sigma + \frac{\pi}{\beta}\right)  \geq 0.
$$ 
Consequently, by using that $\phi$ is decreasing and $\tilde{\eta} \leq 0$, we obtain, for any $\zeta \geq \zeta_1^*$,
\begin{align}
\label{lemestetae1}
\tilde{\eta}(\zeta) &\geq
 \int_{\zeta}^{\zeta + \frac{\pi}{\beta}} G_N(\sigma - \zeta)\phi((f + \tilde{\eta})(\sigma)) d\sigma 
 +  m^2 \int_{\zeta}^{\infty} G_N(\sigma - \zeta) e^{-2m\sigma} (\tilde{\eta} +f) (\sigma ) d\sigma \nonumber\\
&\geq  \int_{\zeta}^{\zeta + \frac{\pi}{\beta}} G_N(\sigma - \zeta) \phi(f(\sigma)) d\sigma  - m^2 (1+P_{N, \varepsilon_0})  \int_{\zeta}^{\infty} |G_N(\sigma - \zeta)| e^{-2m\sigma} f (\sigma ) d\sigma .
\end{align}
Using the explicit forms of $G_N$ and $f$, a direct computation allows us to estimate the second term
\begin{equation}
m^2 (1+P_{N, \varepsilon_0})  \int_{\zeta}^{\infty} |G_N(\sigma - \zeta)| e^{-2m\sigma} f (\sigma ) d\sigma 
\leq (1+P_{N,\varepsilon_0})  m^2  \frac{D_p}{(4m+\alpha /2) \beta} e^{-4m \zeta}.
\end{equation}
In order to estimate the first term on the right hand side of \eqref{lemestetae1}, we use that $x\mapsto \phi(x)/x^2$ is decreasing, and therefore for any $y \geq x > 0$,  
\begin{equation}
\frac{\phi(x)}{x^2} \geq \frac{\phi(y)}{y^2} \,,
\end{equation}
 which implies
\begin{equation}
\phi(f(\sigma)) \geq \phi(f(\zeta)) \left( \frac{f(\sigma)}{f(\zeta)} \right)^2 = \phi(f(\zeta)) e^{-4m(\sigma - \zeta)} .
\end{equation}

Thus, inserting the two previous estimates into \eqref{lemestetae1}, we obtain for any $\zeta \geq \zeta_1^*$
\begin{align}
\tilde{\eta}(\zeta)
&\geq \dfrac{\phi(f(\zeta))}{\beta} \int_{\zeta}^{\zeta + \frac{\pi}{\beta}} e^{-(\frac{\alpha}{2} +4m)(\sigma - \zeta)} \sin(\beta(\sigma - \zeta)) d\sigma   - (1+P_{N,\varepsilon_0}) m^2  \frac{D_p}{(4m+\alpha /2) \beta} e^{-4m \zeta}
\\
&= \dfrac{4\phi(f(\zeta))}{(\alpha + 8m)^2 + 4\beta^2} (1 + e^{-\frac{(\alpha + 8m) \pi}{2\beta}}) - (1+P_{N,\varepsilon_0}) m^2  \frac{D_p}{(4m+\alpha /2) \beta} e^{-4m \zeta}
  \,.
\end{align}
Since $\zeta_1^* \geq \tilde{\zeta}_p$, we have $f(\zeta) \leq C/p$, for any $\zeta \geq \zeta_1^*$. 
Thus,  there exists a constant $C_N>0$, not depending on $p$, such that
$$ 
(1+P_{N,\varepsilon_0}) m^2  \frac{D_p}{(4m+\alpha /2) \beta} e^{-4m \zeta} \leq C_N  p^{1/2} \beta^{-1}   f^2 (\zeta)\leq C_N  p^{-1/2}   f(\zeta) .
$$
Using again that  $\tilde{\eta} \leq 0$ and  $x \mapsto \phi(x)/x$ is decreasing and the definition of $P_N$ (see \eqref{defPN}), we obtain, 
for any $\zeta \geq \tilde{\zeta}_p$ and sufficiently large $p$,
\begin{equation}\label{ren}
|\tilde{\eta}(\zeta)| \leq 
  \dfrac{4|\phi(f(\zeta))|}{((\alpha + 8)^2 + 4\beta^2)f(\zeta)} (1 + e^{-\frac{(\alpha +8)\pi}{2\beta}}) f(\zeta) +  \frac{C}{p^{\frac{1}{2}}} f(\zeta)\leq
\left( P_N +  \frac{C}{p^{\frac{1}{2}}} \right) f(\zeta) < \left(1 + \frac{\varepsilon_0}{2}\right) P_N f(\zeta) \,.
\end{equation}
If $\zeta_1^* > \tilde{\zeta}_p$ , then, by continuity and \eqref{ren}, $|\tilde{\eta}(\zeta)| \leq (1 + \varepsilon_0)P_N |f(\zeta)| $  holds for any $\tilde{\zeta}_p \leq \zeta \leq \zeta_1^*$ sufficiently close to $\zeta_1^*$, a contradiction to the definition of $\zeta_1^*$. Thus 
$\zeta_1^* = \tilde{\zeta}_p$ as desired. 

If $N > 10$,  
using $G_N \geq 0$, the monotonicity of $\phi$, and $\tilde{\eta} \leq 0$ as above, we obtain, for any $\zeta \geq \zeta_1^*$,
\begin{align*}
\tilde{\eta} (\zeta )&\geq  \int^\infty_{\zeta} G_N(\sigma - \zeta)(\phi (f(\zeta)) e^{-4m (\sigma -\zeta)} - m^2 (1+P_{N,\varepsilon_0}) e^{-2m\sigma} f(\sigma ))   d\sigma 
\\
&\geq \phi (f(\zeta)) \int^\infty_{\zeta} G_N(\sigma - \zeta)e^{-4m (\sigma - \zeta)} d\sigma - \frac{1}{2\beta |\frac{\alpha}{2} -\beta +4m|} f^2(\zeta) \,,
\end{align*}
Then, one has 
\begin{equation}
\tilde{\eta} (\zeta) \geq \dfrac{1}{2\beta (\frac{\alpha}{2}+4m -\beta )}  \phi (f(\zeta))    -   O(p^{-1}) f^2(\zeta ) \, .
\end{equation}
Proceeding as above, we find
$$|\tilde{\eta} (\zeta )|\leq  \dfrac{ \phi (f(\zeta)) }{2\beta (\frac{\alpha}{2}+4m -\beta ) f(\zeta )} f(\zeta )+ \frac{C}{p} f^2(\zeta ) 
< 
\left(1+\frac{ \varepsilon_0}{2}\right) P_N f(\zeta ).   $$
And the proof is concluded as in the previous case.
\end{proof}

\begin{rmq} \label{pbou}
In the  Lemma \ref{tildeetapetit}, we proved that 
\begin{equation}
0 \geq \tilde{\eta}(\zeta) \geq  -(1+\varepsilon_0) P_N f(\zeta) \qquad \textrm{for any } \zeta \geq \tilde{\zeta}_p
\end{equation}
which combined with $P_N \leq \frac{1}{2}$ imply 
\begin{equation}
0 \leq \eta \leq f(\zeta) \qquad \textrm{for any } \zeta \geq  \tilde{\zeta}_p \,.
\end{equation}
In the original variables, we have
\begin{equation}
A_{p, N} r^{-\theta} \leq   U^\ast_p (r) \leq A_{p, N} r^{-\theta} (1+ D_p r^2)
\qquad \textrm{for any } r \leq \tilde{c}/\sqrt{p} \,.  
\end{equation}
The importance of this bound is in the estimate on $U^*_p$ on an explicit interval.  
\end{rmq}

\begin{prop}
\label{LNTpinfty}
For any fixed $i\in \N$, we have
$$R_p^i \rightarrow 0, \ as\ p\rightarrow \infty .$$
\end{prop}

\begin{proof}
Assume $\tilde{c}$ is as in Lemma \ref{l:boct} and denote $\tilde{r}_p = \tilde{c}/\sqrt{p}$ and $\tilde{\zeta}_p = - m^{-1} \ln \tilde{r}_p$. 
Then, Remark \ref{pbou} holds on the interval $(0, \tilde{r}_p]$.
As above, we denote by $C_N$ constants depending on $N$ but not on $p$. 

First assume $N\leq 10$. 
Observe that  \eqref{phes}, implies 
$|\phi (z)| \leq C_N p^2 z^2$ for $z \leq c_N/p$ and consequently Lemma \ref{tildeetapetit} yields  $|\tilde{g}(\zeta)| \leq C_n p^2 e^{-4m\zeta}$ for any $\zeta \geq \tilde{\zeta}_p$.
Then, taking the derivative of the representation formula \eqref{repform}, using that $\tilde{\eta} \leq 0$, asymptotics \eqref{LNTlem1asy1}, and the definition 
of $\tilde{\zeta}_p$ (see \eqref{relrzeta}), we have
\begin{equation}\label{ovn}
\begin{aligned}
\tilde{\eta}^\prime (\tilde{\zeta}_p  )&= \frac{\alpha}{2}\tilde{\eta}(\tilde{\zeta}_p) - e^{(\alpha /2) \tilde{\zeta}_p} \int_{\tilde{\zeta}_p}^\infty e^{-(\alpha /2 )\sigma } \cos (\beta (\sigma -\tilde{\zeta}_p)) \tilde{g}(\sigma ) d\sigma \\
&\leq C_N p^2 e^{(\alpha /2) \tilde{\zeta}_p} \int_{\tilde{\zeta}_p}^\infty e^{-(\alpha /2)\sigma} e^{-4m\sigma}d\sigma  \\
&\leq C_N  p^{3/2} e^{-4m\tilde{\zeta}_p}  \leq \frac{C_N}{\sqrt{p}}.
\end{aligned}
\end{equation}
The same estimate holds true for $N>10$ since 
\begin{align*}
 \int_{\tilde{\zeta}_p}^\infty e^{-(\alpha /2)\sigma} \cosh (\beta (\sigma -\tilde{\zeta}_p) e^{-4m\sigma}d\sigma 
 &\leq  \int_{\tilde{\zeta}_p}^\infty e^{-(\alpha /2)\sigma} e^{\beta (\sigma -\tilde{\zeta}_p)} e^{-4m\sigma}d\sigma \\
&\leq C_N p^{-1/2} e^{-(\alpha /2) \tilde{\zeta}_p }e^{-4m \tilde{\zeta}_p}.
\end{align*}
Thus, we have, using that $U_p^\ast (r)= A_{p,N} r^{-\theta} (\tilde{\eta}(\zeta ) +1+D_p e^{-2m\zeta}) $, \eqref{LNTlem1asy1}, Lemma \ref{tildeetapetit}, \eqref{ovn}, and the 
definition of $\tilde{r}_p$
\begin{equation} \label{desu}
\begin{aligned}
|(U_p^\ast )^\prime (\tilde{r}_p)|&= A_{p,N} \tilde{r}_p^{-\theta} \left| (-\theta) \left( \frac{\tilde{\eta} (\tilde{\zeta}_p) +1}{\tilde{r}_p} +D_p \tilde{r}_p\right) - \frac{\tilde{\eta}^\prime (\tilde{\zeta}_p)}{m \tilde{r}_p} +2D_p \tilde{r}_p \right|\\
&\leq C_N (\tilde{r}_p +   (p \tilde{r}_p)^{-1} (1+|\tilde{\eta} (\tilde{\zeta}_p)|) + |\tilde{\eta}^\prime (\tilde{\zeta}_p)| \tilde{r}_p^{-1}p^{-1/2} )  \\
&\leq C_N p^{-1/2} \rightarrow 0,\ as\ p\rightarrow \infty . 
\end{aligned}
\end{equation}
In addition,  Remark \ref{pbou} implies
\begin{equation}
A_{p, N} \left(\frac{p}{\tilde{c}^2}\right)^{\frac{1}{p-1}} \leq  U^*_p(\tilde{r}_p) \leq A_{p, N} \left(\frac{p}{\tilde{c}^2}\right)^{\frac{1}{p-1}} \left(1 + D_p\frac{\tilde{c}^2}{p}\right) \,,
\end{equation}
and consequently $U^*_p(\tilde{r}_p) \to 1$ as $p \to \infty$. Also, we have
\begin{equation}
A_{p, N}^{p + 1} \left(\frac{p}{\tilde{c}^2}\right)^{\frac{p+1}{p-1}} \leq  (U^*_p(\tilde{r}_p))^{p + 1} \leq A_{p, N}^{p + 1} \left(\frac{p}{\tilde{c}^2}\right)^{\frac{p + 1}{p-1}} 
\left(1 + D_p\frac{\tilde{c}^2}{p}\right)^{p + 1}.
\end{equation}
Since $A_{p, N}^{p + 1} \approx (p - 1)^{-\frac{p+ 1}{p -1}}$, we obtain that $(U^*_p(\tilde{r}_p))^{p + 1} \approx 1$, and therefore $(U^*_p(\tilde{r}_p))^{p + 1}/(p + 1) \approx p^{-1}$ as 
$p \to \infty$.

Next, we will prove more precise estimate.
Since the function 
\begin{equation}
r \mapsto E(r) = \frac{((U_p^\ast)^\prime (r))^2}{2} - \frac{1}{2} (U_p^\ast)^2(r) + \frac{(U_p^\ast )^{p+1} (r)}{p+1}
\end{equation}
is non-increasing, then by the above estimates,  one has, for any $r \geq \tilde{r}_p$,
\begin{equation}
((U_p^\ast)^\prime  (\tilde{r}_p ) )^2- \frac{1}{2} + \mu_p \geq  E(\tilde{r}_p) 
\geq E(r) \geq - \frac{(U_p^\ast)^2(r)}{2}  \,,
\end{equation}
where $\mu_p \to 0$ as $p \to \infty$. 
Hence, from \eqref{desu} follows
\begin{equation}
1 - (U_p^\ast)^2(r) \leq 2 \mu_p ,
\end{equation}
and therefore $(1 - (U_p^\ast)^2(r))_{+} \to 0$ as $p \to \infty$, where $h_+ = \max \{h, 0\}$ denotes the positive part of a function $h$.
On the other hand, if there is $\varepsilon^* > 0$ and $r^*_p > \tilde{r}_p$ such that $U^*_p(r^*_p) \geq 1 + \varepsilon^*$, then 
\begin{equation}
E(r^*_p) \geq - \frac{1}{2} (U_p^\ast)^2(r) + \frac{(U_p^\ast )^{p+1} (r)}{p+1}  \to \infty \qquad \textrm{as } p \to \infty \,,
\end{equation}
a contradiction to $E(r^*_p) \leq E(\tilde{r}_p) \leq C_N$. 

Overall  we proved that $|U_p^\ast (r)-1| \rightarrow 0$, for all $r\geq \tilde{r}_p$. Recall that $w(r)=r^{\frac{N-1}{2}}(U_p^\ast (r) - 1)$ (see \eqref{eqw}) satisfies 
$$
w^{\prime \prime} + \left(\frac{(U_p^\ast )^p - U_p^\ast}{U_p^\ast-1} - \frac{(N-1)(N-3)}{r^2} \right) w=0.
$$
Fix  $a > 0$ and denote $I_a := \left[\frac{a}{4}, a\right]$.
Choose any $r \in I_a$. Since $|U^\ast_p (r)-1|\rightarrow 0$ locally uniformly, we have
$$
\dfrac{(U^\ast_p (r))^p - U^\ast_p (r)}{U^\ast_p (r)-1}\geq p/2 , \qquad \textrm{on } I_a.
$$
Fix a large $A > 0$ depending on $a$ as specified below.  Then, for sufficiently large $p > 2^\ast -1$ depending on $a$ and $A$, one has  
$$ 
\dfrac{(U^\ast_p (r))^p - U^\ast_p (r)}{U^\ast_p (r)-1} -\dfrac{(N-1)(N-3)}{4r^2}  \geq A - C_{N, a} \qquad \textrm{for any } \quad r \in I_a := \left[\frac{a}{4}, a\right].
$$
Thus, given $a > 0$ and an integer $i > 0$,  we choose $A$ large enough such that a solution of  the  equation $z'' + (A - C_{N, a} )z = 0$ has at least $i + 2$ zeros on $I_a$. By 
Sturm-Picone  oscillation theorem, the function $w$ has at least $i + 1$ zeros on $I_a$. 
Consequently, $U_p^* (r) = 1$ has at least $i + 1$ solutions on  $I_a$, and therefore $U_p^*$ has at least $i$
critical points on $I_a$. In a different notation, for any $j \in \{1, \cdots, i\}$ and any $a > 0$, one has $R^j_p < a$, for any sufficiently large $p > 2^\ast -1$.   
\end{proof}

\begin{rmq}\label{deub}
By \eqref{desu}, we have 
\begin{equation}
\left|(U^\ast_p)'(r) +\theta A_{p,N}r^{-1-\theta} \right|\leq  
\left| A_{p,N} r^{-\theta} \left( 
(-\theta) \left( \frac{\tilde{\eta} (\zeta)}{r} +D_p r \right) - \frac{\tilde{\eta}^\prime (\zeta)}{m r} +2D_p r 
\right)
 \right| \,.
\end{equation}
By Remark \ref{pbou} and \eqref{ovn}, one has 
$|\eta(\zeta)| \leq f(\zeta) \leq C_N r^2$ and $|\eta'(\zeta)| \leq C_N \sqrt{p} r^2 $
for any $r \leq \frac{\tilde{c}}{\sqrt{p}}$.
Thus, 
\begin{equation}
\left|(U^\ast_p)'(r) +\theta A_{p,N}r^{-1-\theta} \right|\leq C_N r^{1 - \theta},  \textrm{ for any } r \leq \frac{\tilde{c}}{\sqrt{p}} \,.
\end{equation}
\end{rmq}

\begin{prop}
\label{LNTcontmp}
For any $i\in \N$, the function $p\rightarrow R_p^i$ is continuous.
\end{prop}

\begin{proof}
Let $p^\ast > 2^\ast -1$. 
 Fix any open interval $I_0 = (A, B)$ such that $0 < A < B < \infty$ and without loss of generality assume that $A < \tilde{c}/(2p)$. Then, by Remark \ref{pbou}, there is $\delta > 0$ such that, for any 
$p \in (p^\ast - \delta, p^* + \delta)$, one has
\begin{equation}
|U^*_p (A)| \leq C_N \,.
\end{equation}
If $r \leq R_p^1$, since $U_p^\ast$ is decreasing and positive (see \cite[Theorem A.3]{miya1}) 
on $(0,R_p^1)$, we have $|U^*_p (r)| \leq C_N$, for any $r\in (A,B)$. If $r >R_p^1$, we use the fact that the functional 
$$E(r)=\frac{((U_p^\ast)^\prime (r))^2}{2} -\frac{(U_p^\ast (r))^2}{2} +\frac{(U_p^\ast (r))^{p+1}}{p+1},$$
is decreasing. 
Since $U_p^\ast (R^1_p) \leq 1$ by \cite[Lemma 4.8]{miya1}, 
this implies for any $r\geq R_p^1$ that 
$$\frac{(U_p^\ast (r))^{p+1}}{p+1} -\frac{(U_p^\ast (r))^2}{2}  \leq  \frac{(U_p^\ast (R_p^1))^{p+1}}{p+1}\leq \frac{1}{p+1}.  $$
Thus, also in this case, we have $|U^*_p| \leq C_N$ on $(A, B)$. Overall, we showed that 
\begin{equation*}
\sup_{p \in (p^* - \delta, p^* + \delta)} \sup_{(A, B)} U_p^\ast \leq C(A) \,.
\end{equation*}
Then, elliptic regularity theory implies  that, for any $q > 1$,
\begin{equation}
\label{LNTconte1}
\|U^\ast_p \|_{W^{3, q}(I_0)} \leq C(N, q, A, B - A,  p^\ast, \delta) ,\qquad \textrm{for any } p \in (p^* - \delta, p^* + \delta) \,.
\end{equation}

Let $\alpha_0 \in (0, 1)$. We choose $q_0>0$ large enough such that $W^{3, q_0} (I_0) \hookrightarrow C^{2 + \alpha_0}(I_0)$.  Let $(p_n)$ be a sequence such that $p_n \to p^*$ when $n\rightarrow \infty$. Thanks to \eqref{LNTconte1}, using Arzela-Ascoli Theorem, there exists a subsequence $(p_n)$ such that $U^\ast_{p_n} \to w$, as $n\to \infty$, in $C^{2}(I_0)$. From the uniform bound \eqref{LNTconte1} follows
\begin{equation*}
|(U^\ast_{p_n})^{p_n} (s) - w^{p^*}(s)| \leq |(U^\ast_{p_n})^{p_n} (s) - (U^\ast_{p_n})^{p^\ast} (s)| + |(U^\ast_{p_n})^{p^\ast} (s)  - w^{p^*}(s)|,
\end{equation*}
we deduce that $w$ satisfies the equation
\begin{equation*}
- \Delta w + w = w^{p^*}, \qquad \textrm{in } I_0 \,.
\end{equation*}
Since $I_0$ is an arbitrary compact interval, proceeding as above and using standard diagonal arguments, 
we obtain the existence of a subsequence $(p_n)_n$, $p_n \in (p^\ast -\delta , p^\ast +\delta )$, for all $n\in \N$, such that $U^\ast_{p_n} \to w$, as $n\to \infty$, in $C^{2}_{loc}((0,\infty ))$, for some function $w$ satisfying
\begin{equation*}
- \Delta w + w = w^{p^*} \qquad \textrm{in } (0, \infty) \,.
\end{equation*} 
Next, we claim that $w$ is in fact equal to $U_{p^\ast}^\ast$. Using the uniqueness of solution to \eqref{LNTwhole} (see Theorem \ref{LNTthmmiya}), it is sufficient to show that 
\begin{equation}
\label{LNTconte2}
\lim_{r\rightarrow 0^+} r^{\theta_{p^\ast} } w(r)=A_{p^\ast, N },
\end{equation}
where $\theta_{p^\ast}= \dfrac{2}{p-1}$. However if $p > 2$, by Remark \ref{pbou} for any $\varepsilon > 0$, there is $r_0(\varepsilon)$ independent of
$ p \in (p^\ast - \delta, p^* + \delta)$ such that 
\begin{equation*}
A_{p, N} r^{-\theta_p} \leq U_p^\ast (r) \leq   A_{p, N} r^{-\theta_p} + \varepsilon, \qquad \textrm{for all }\quad r \in (0, r_0(\varepsilon)) \,.
\end{equation*} 
Clearly $A_{p_n, N}\rightarrow A_{p^\ast, N}$, $\theta_{p_n}\rightarrow \theta_{p^\ast}$, when $n\to \infty$ and
using that $U_{p_n}^\ast \to w$ in 
$C^2_{\textrm{loc}} ((0, r_0(\varepsilon))) $, we obtain
\begin{equation*}
A_{p^*, N} r^{-\theta_{p^\ast}} \leq w(r) \leq  A_{p^*, N} r^{-\theta_{p^\ast}} + \varepsilon , \qquad \textrm{for all }\quad r \in (0, r_0(\varepsilon)) \,.
\end{equation*} 
Since $\varepsilon > 0$ is arbitrary, we conclude that \eqref{LNTconte2} holds, and therefore $w=U_{p^\ast}^\ast$
by uniqueness. Hence, 
\begin{equation}
\label{KSconte3}
U_p^\ast \rightarrow U_{p^\ast}^\ast ,\textrm{ as }\ p\to p^\ast,\ \textrm{ in }\ C_{loc}^2((0,\infty)).
\end{equation}
Finally, we prove the continuity of the function $p\rightarrow R_p^i$. In the following, we assume that $R_{p^*}^i$ is a local minimum of $U_{p^*}^\ast$, 
the case of local maximum follows analogously. Note that $U_{p^*}^*(R_{p^*}^i) \neq 1$, 
otherwise $U^* \equiv 1$, and 
we have a contradiction to the uniqueness of the initial value problem. Thus, for any sufficiently small $\bar{\varepsilon} > 0$, we obtain
\begin{equation*}
U^\ast_{p^\ast}(R_{p^*}^i - \bar{\varepsilon}) > U^\ast_{p^\ast}(R_{p^*}^i) \qquad \textrm{and } \quad U^\ast_{p^\ast}(R_{p^*}^i + \bar{\varepsilon}) > U^\ast_{p^\ast}(R_{p^*}^i) \,.
\end{equation*}
Then \eqref{KSconte3} yields that, for $p$ sufficiently close to $p^\ast$, there exists a local minimizer $q_p$ of $U^\ast_{p}$ 
in $(R_{p^*}^i - \bar{\varepsilon}, R_{p^*}^i + \bar{\varepsilon})$. Since $\bar{\varepsilon} > 0$ was arbitrary,
for each $p$ close to $p^*$, there is a local minimizer $q_p$ of $(U^\ast_{p})'(q_p) = 0$
such that
\begin{equation}
\lim_{p \to p^\ast} q_p =  R_{p^*}^i \,.
\end{equation}
On the other hand, if there exists a sequence $(p_n)_{n \in \N}$ such that $p_n \to p^*$ and $(q_{p_n})_{n \in \N}$
converges to $R^*$, then by \eqref{KSconte3}, one has $(U^*_{p^*})' (R^*)  = 0$. Equivalently $R^* = R^j_{p^*}$ for some $j$. Thus, 
we proved that the critical points of $U^*_p$ concentrate around critical points of   $U^*_{p^*}$ and in arbitrary small neighborhood of $R_{p^*}^i$ there is a 
critical point of $U^*_p$.   

To finish the proof, we show that in a small neighborhood of $R^j_{p^*}$, there exists at most one critical point of $U^*_p$.
For a contradiction, assume that there exists a sequence $(p_n)_{n \in \N}$ such that $p_n \to p^*$ and both sequences $(q_{p_n})_{n \in \N}$, 
$(q_{p_n}')_{n \in \N}$ converge to $q^*$. Then by the mean value theorem, there exists $s_{p_n}$ between $q_{p_n}$ and 
$q_{p_n}'$ such that $(U_{p^\ast}^\ast)''(s_{p_n}) = 0$. By passing to the limit, one has $(U^*_{p^*})' (q^*)  = (U^*_{p^*})'' (q^*)  = 0$, a contradiction to the fact that every critical point is either strict minimizer or strict maximizer (otherwise by the uniqueness
of solutions to initial value problems, $U^*_{p^*}$ is constant). 

Overall, we proved that in each neighborhood of $R^i_{p^*}$, there exists exactly one critical point of 
$U^*_\lambda$ and the proof is finished.
\end{proof}

We are now in position to prove Theorem \ref{LNTmainthm}.

\begin{proof}[Proof of Theorem \ref{LNTmainthm}.]  By assumptions, 
we know that, for any $i\geq \tilde{i}$, $R_{\tilde{p}}^{i}>R$. On the other hand, by Proposition \ref{LNTpinfty},  for any $i\in \N$, $\lim_{p\rightarrow \infty}R_p^i < R$. Since the function $p\rightarrow R_i^p$, for any $i\in \N$, is continuous by Proposition \ref{LNTcontmp}, we deduce that there exists $p^i > \tilde{p}$ such that $R_{p^i}^i =R$. This concludes the proof.
\end{proof}

\section{Proof of Proposition \ref{LNTmorse}}

First, we show that $U_{p^i}^\ast$ has a finite (resp. infinite) Morse index provided that $p^i >p_{JL}$ (resp. $2^\ast -1 <p^i <p_{JL}$), i.e. we prove Proposition \ref{LNTmorse}.

\begin{proof}[Proof of Proposition \ref{LNTmorse}.]
Fix $p := p^i > p_{JL}$. Then, for sufficiently small $\varepsilon_0 > 0$, one has 
\begin{equation}
p\theta (N-2-\theta) < (1 - \varepsilon_0) \frac{(N - 2)^2}{4} \,.
\end{equation}
Due to boundary conditions in \eqref{LNTwhole}, there exists $r_0\in (0,1)$ such that, for any $r\in (0,r_0)$,
$$
p (U_p^\ast)^{p-1} (r)-1 \leq 
p (U_p^\ast)^{p-1} (r) \leq 
\dfrac{p\theta (N-2-\theta)}{r^2}(1+\varepsilon_0) \leq  \dfrac{(N-2)^2}{4 r^2} (1 - \varepsilon_0^2). 
$$
Let $\chi_0 \in C^1 (\R^N)$ be a cut-off function such that
 $\chi_0 (r)=
 \begin{cases}
 1, & \textrm{if } r\in (0, r_0 /2)\\ 
 0,& \textrm{if }  r> r_0 
 \end{cases}$, and let $\chi_1 =1-\chi_0$. We take $\phi \in H^1_{rad}(B_R (0))$ such that $\phi '(R)=0$. Then we have, thanks to the Hardy inequality,
\begin{align*}
\mathcal{J}(\phi )=\int_{B_R (0)} (|\nabla \phi |^2 - (p (U_p^\ast)^{p-1} (r)-1) \phi^2 dx &= \int_{B_R (0)} (|\nabla \phi |^2 -(\chi_0 +\chi_1) (p (U_p^\ast)^{p-1} (r)-1)  \phi^2 dx\\
&\geq  \int_{B_R (0)} ((1 - \varepsilon_0^2) |\nabla \phi |^2 -\chi_0 \dfrac{(N-2)^2}{4 r^2} (1 - \varepsilon_0^2)  \phi^2 dx \\
&+ \int_{B_R (0)} (\varepsilon_0^2 |\nabla \phi |^2 -\chi_1 (p (U_p^\ast)^{p-1} (r)-1)  \phi^2 dx\\
&\geq \int_{B_R (0)} (\varepsilon_0^2 |\nabla \phi |^2 -\chi_1 (p (U_p^\ast)^{p-1} (r)-1)  \phi^2 dx \,.
\end{align*}
Since $|p U_p^\ast )^{p-1}(r)-1|\leq C$, for $r\in (r_0 /2 , R)$, and the operator $-\varepsilon_0^2\Delta  -\chi_1 (p (U_p^\ast)^{p-1} (r)-1)  $ with Neumann boundary condition has a finite number of negative eigenvalues, we conclude that $m(U_p^\ast )<\infty$.

 Next assume that $2^\ast -1 < p< p_{JL}$. As above, using boundary condition in \eqref{LNTwhole}, one has that, for some small $\varepsilon_0>0$, there exists $r_0$ such that, for all $r\in (0,r_0)$,
\begin{equation}
\label{LNTup}
p (U_p^\ast)^{p-1} (r)-1 \geq  \left(\dfrac{(N-2)^2}{4 }+\varepsilon_0^2 \right) 
\frac{1}{r^2}. 
\end{equation}
Next, we define $f_j (r)=f(r) \tilde{\chi}_j (r)$, where 
$$
\tilde{\chi}_j (r)=
\begin{cases}1,& \text{ if $r\in [r_{j+1},r_j]$,}\\ 
0, & \text{ elsewhere } ,
\end{cases}
\qquad 
r_j=e^{-2\pi j/\varepsilon_0}
$$
and $f(r)=r^{-(N-2)/2} \sin (\varepsilon_0 \log r /2)$. Notice that $f_j$ and $f_k$ have disjoint supports for 
$j \neq k$, and therefore they are linearly independent. Moreover, $f_j$  is a solution of 
$$
- f_j'' - \frac{N - 1}{r} f_j' -  \left(\dfrac{(N-2)^2}{4}+\dfrac{\varepsilon_0^2}{4}\right) \dfrac{1}{r^2} f_j=0,\quad r\in (r_{j+1},r_j).
$$
Since $f_j(r_j) = f_j(r_{j + 1}) = 0$ we have that $f_j \in W^{1, 2}((0, \infty))$ and by \eqref{LNTup}
\begin{align}
\mathcal{J}(f_j) \geq 
\int_{r_{j+1}}^{r_j} \left(|f_j'|^2 -  \left(\dfrac{(N-2)^2}{4}+\varepsilon_0^2\right)\dfrac{1}{r^2} f_j^2\right)r^{N-1} dr = -\dfrac{3}{4}\varepsilon_0^2 \int_{r_{j+1}}^{r_j} \dfrac{1}{r^2} f_j^2 dx <0 \, .
\end{align}
Thus the Morse index of $U^*_{p}$ is infinite.

\end{proof}



\begin{bibdiv}
\begin{biblist}

\bib{MR3801293}{article}{
      author={Ao, Weiwei},
      author={Chan, Hardy},
      author={Wei, Juncheng},
       title={Boundary concentrations on segments for the {L}in-{N}i-{T}akagi
  problem},
        date={2018},
        ISSN={0391-173X},
     journal={Ann. Sc. Norm. Super. Pisa Cl. Sci. (5)},
      volume={18},
      number={2},
       pages={653\ndash 696},
      review={\MR{3801293}},
}

\bib{MR2812575}{article}{
      author={Ao, Weiwei},
      author={Musso, Monica},
      author={Wei, Juncheng},
       title={On spikes concentrating on line-segments to a semilinear
  {N}eumann problem},
        date={2011},
        ISSN={0022-0396},
     journal={J. Differential Equations},
      volume={251},
      number={4-5},
       pages={881\ndash 901},
         url={https://doi.org/10.1016/j.jde.2011.05.009},
      review={\MR{2812575}},
}

\bib{BCF}{article}{
      author={Bonheure, Denis},
      author={Casteras, Jean-Baptiste},
      author={F\"{o}ldes, Juraj},
       title={Singular radial solutions for the keller-segel equation in high
  dimension},
        date={2018},
     journal={Preprint arXiv:1808.06990},
}

\bib{BonheureGrossiNorisTerracini2015}{article}{
      author={Bonheure, Denis},
      author={Grossi, Massimo},
      author={Noris, Benedetta},
      author={Terracini, Susanna},
       title={Multi-layer radial solutions for a supercritical neumann
  problem},
        date={2016},
        ISSN={0022-0396},
     journal={Journal of Differential Equations},
      volume={261},
      number={1},
       pages={455 \ndash  504},
  url={http://www.sciencedirect.com/science/article/pii/S0022039616001212},
}

\bib{BonheureGrumiauTroestler2015}{article}{
      author={Bonheure, Denis},
      author={Grumiau, Christopher},
      author={Troestler, Christophe},
       title={Multiple radial positive solutions of semilinear elliptic
  problems with {N}eumann boundary conditions},
        date={2016},
        ISSN={0362-546X},
     journal={Nonlinear Anal.},
      volume={147},
       pages={236\ndash 273},
         url={https://doi.org/10.1016/j.na.2016.09.010},
      review={\MR{3564729}},
}

\bib{buddnorbury}{article}{
      author={Budd, Chris},
      author={Norbury, John},
       title={Semilinear elliptic equations and supercritical growth},
        date={1987},
        ISSN={0022-0396},
     journal={J. Differential Equations},
      volume={68},
      number={2},
       pages={169\ndash 197},
         url={http://dx.doi.org/10.1016/0022-0396(87)90190-2},
      review={\MR{892022}},
}

\bib{dancershusen}{article}{
      author={Dancer, E.~N.},
      author={Yan, Shusen},
       title={Multipeak solutions for a singularly perturbed {N}eumann
  problem},
        date={1999},
        ISSN={0030-8730},
     journal={Pacific J. Math.},
      volume={189},
      number={2},
       pages={241\ndash 262},
         url={https://doi.org/10.2140/pjm.1999.189.241},
      review={\MR{1696122}},
}

\bib{MR3262455}{article}{
      author={del Pino, Manuel},
      author={Mahmoudi, Fethi},
      author={Musso, Monica},
       title={Bubbling on boundary submanifolds for the {L}in-{N}i-{T}akagi
  problem at higher critical exponents},
        date={2014},
        ISSN={1435-9855},
     journal={J. Eur. Math. Soc. (JEMS)},
      volume={16},
      number={8},
       pages={1687\ndash 1748},
         url={https://doi.org/10.4171/JEMS/473},
      review={\MR{3262455}},
}

\bib{MR2114411}{article}{
      author={del Pino, Manuel},
      author={Musso, Monica},
      author={Pistoia, Angela},
       title={Super-critical boundary bubbling in a semilinear {N}eumann
  problem},
        date={2005},
        ISSN={0294-1449},
     journal={Ann. Inst. H. Poincar\'{e} Anal. Non Lin\'{e}aire},
      volume={22},
      number={1},
       pages={45\ndash 82},
         url={https://doi.org/10.1016/j.anihpc.2004.05.001},
      review={\MR{2114411}},
}

\bib{dolbeaultflores}{article}{
      author={Dolbeault, Jean},
      author={Flores, Isabel},
       title={Geometry of phase space and solutions of semilinear elliptic
  equations in a ball},
        date={2007},
        ISSN={0002-9947},
     journal={Trans. Amer. Math. Soc.},
      volume={359},
      number={9},
       pages={4073\ndash 4087},
         url={http://dx.doi.org/10.1090/S0002-9947-07-04397-8},
      review={\MR{2309176}},
}

\bib{MR2779463}{book}{
      author={Dupaigne, Louis},
       title={Stable solutions of elliptic partial differential equations},
      series={Chapman \& Hall/CRC Monographs and Surveys in Pure and Applied
  Mathematics},
   publisher={Chapman \& Hall/CRC, Boca Raton, FL},
        date={2011},
      volume={143},
        ISBN={978-1-4200-6654-8},
         url={https://doi.org/10.1201/b10802},
      review={\MR{2779463}},
}

\bib{MR2334996}{article}{
      author={Esposito, Pierpaolo},
       title={Estimations \`a l'int\'{e}rieur pour un probl\`eme elliptique
  semi-lin\'{e}aire avec non-lin\'{e}arit\'{e} critique},
        date={2007},
        ISSN={0294-1449},
     journal={Ann. Inst. H. Poincar\'{e} Anal. Non Lin\'{e}aire},
      volume={24},
      number={4},
       pages={629\ndash 644},
         url={https://doi.org/10.1016/j.anihpc.2006.04.004},
      review={\MR{2334996}},
}

\bib{MR0153960}{article}{
      author={Gel'fand, I.~M.},
       title={Some problems in the theory of quasilinear equations},
        date={1963},
        ISSN={0065-9290},
     journal={Amer. Math. Soc. Transl. (2)},
      volume={29},
       pages={295\ndash 381},
      review={\MR{0153960}},
}

\bib{Gierer-Meinardt}{article}{
      author={Gierer, A.},
      author={Meinhardt, H.},
       title={A theory of biological pattern formation},
        date={197201},
     journal={Biological Cybernetics},
      volume={12},
       pages={30\ndash 39},
}

\bib{guowei}{article}{
      author={Guo, Zongming},
      author={Wei, Juncheng},
       title={Global solution branch and {M}orse index estimates of a
  semilinear elliptic equation with super-critical exponent},
        date={2011},
        ISSN={0002-9947},
     journal={Trans. Amer. Math. Soc.},
      volume={363},
      number={9},
       pages={4777\ndash 4799},
         url={http://dx.doi.org/10.1090/S0002-9947-2011-05292-X},
      review={\MR{2806691}},
}

\bib{josephlundgren}{article}{
      author={Joseph, D.~D.},
      author={Lundgren, T.~S.},
       title={Quasilinear {D}irichlet problems driven by positive sources},
        date={1972/73},
        ISSN={0003-9527},
     journal={Arch. Rational Mech. Anal.},
      volume={49},
       pages={241\ndash 269},
         url={http://dx.doi.org/10.1007/BF00250508},
      review={\MR{0340701}},
}

\bib{MR3606457}{article}{
      author={Lee, Youngae},
      author={Seok, Jinmyoung},
       title={Multiple interior and boundary peak solutions to singularly
  perturbed nonlinear {N}eumann problems under the {B}erestycki-{L}ions
  condition},
        date={2017},
        ISSN={0025-5831},
     journal={Math. Ann.},
      volume={367},
      number={1-2},
       pages={881\ndash 928},
         url={https://doi.org/10.1007/s00208-016-1412-3},
      review={\MR{3606457}},
}

\bib{LNT}{article}{
      author={Lin, C.-S.},
      author={Ni, W.-M.},
      author={Takagi, I.},
       title={Large amplitude stationary solutions to a chemotaxis system},
        date={1988},
        ISSN={0022-0396},
     journal={J. Differential Equations},
      volume={72},
      number={1},
       pages={1\ndash 27},
         url={https://doi.org/10.1016/0022-0396(88)90147-7},
      review={\MR{929196}},
}

\bib{MR974610}{incollection}{
      author={Lin, Chang~Shou},
      author={Ni, Wei-Ming},
       title={On the diffusion coefficient of a semilinear {N}eumann problem},
        date={1988},
   booktitle={Calculus of variations and partial differential equations
  ({T}rento, 1986)},
      series={Lecture Notes in Math.},
      volume={1340},
   publisher={Springer, Berlin},
       pages={160\ndash 174},
         url={http://dx.doi.org/10.1007/BFb0082894},
      review={\MR{974610}},
}

\bib{MR2395219}{article}{
      author={Malchiodi, Andrea},
       title={Concentrating solutions of some singularly perturbed elliptic
  equations},
        date={2008},
        ISSN={1673-3452},
     journal={Front. Math. China},
      volume={3},
      number={2},
       pages={239\ndash 252},
         url={https://doi.org/10.1007/s11464-008-0015-z},
      review={\MR{2395219}},
}

\bib{merlepeletier}{article}{
      author={Merle, F.},
      author={Peletier, L.~A.},
       title={Positive solutions of elliptic equations involving supercritical
  growth},
        date={1991},
        ISSN={0308-2105},
     journal={Proc. Roy. Soc. Edinburgh Sect. A},
      volume={118},
      number={1-2},
       pages={49\ndash 62},
         url={http://dx.doi.org/10.1017/S0308210500028882},
      review={\MR{1113842}},
}

\bib{miya2}{article}{
      author={Miyamoto, Yasuhito},
       title={Classification of bifurcation diagrams for elliptic equations
  with exponential growth in a ball},
        date={2015},
        ISSN={0373-3114},
     journal={Ann. Mat. Pura Appl. (4)},
      volume={194},
      number={4},
       pages={931\ndash 952},
         url={http://dx.doi.org/10.1007/s10231-014-0404-8},
      review={\MR{3357688}},
}

\bib{miya1}{article}{
      author={Miyamoto, Yasuhito},
       title={Structure of the positive radial solutions for the supercritical
  {N}eumann problem {$\varepsilon^2\Delta u-u+u^p=0$} in a ball},
        date={2015},
        ISSN={1340-5705},
     journal={J. Math. Sci. Univ. Tokyo},
      volume={22},
      number={3},
       pages={685\ndash 739},
      review={\MR{3408072}},
}

\bib{NT1}{article}{
      author={Ni, Wei-Ming},
      author={Takagi, Izumi},
       title={On the shape of least-energy solutions to a semilinear {N}eumann
  problem},
        date={1991},
        ISSN={0010-3640},
     journal={Comm. Pure Appl. Math.},
      volume={44},
      number={7},
       pages={819\ndash 851},
         url={https://doi.org/10.1002/cpa.3160440705},
      review={\MR{1115095}},
}

\bib{NT2}{article}{
      author={Ni, Wei-Ming},
      author={Takagi, Izumi},
       title={Locating the peaks of least-energy solutions to a semilinear
  {N}eumann problem},
        date={1993},
        ISSN={0012-7094},
     journal={Duke Math. J.},
      volume={70},
      number={2},
       pages={247\ndash 281},
         url={https://doi.org/10.1215/S0012-7094-93-07004-4},
      review={\MR{1219814}},
}

\bib{MR1707889}{article}{
      author={Rey, Olivier},
       title={An elliptic {N}eumann problem with critical nonlinearity in
  three-dimensional domains},
        date={1999},
        ISSN={0219-1997},
     journal={Commun. Contemp. Math.},
      volume={1},
      number={3},
       pages={405\ndash 449},
         url={https://doi.org/10.1142/S0219199799000158},
      review={\MR{1707889}},
}

\bib{MR1968337}{article}{
      author={Rey, Olivier},
       title={The question of interior blow-up-points for an elliptic {N}eumann
  problem: the critical case},
        date={2002},
        ISSN={0021-7824},
     journal={J. Math. Pures Appl. (9)},
      volume={81},
      number={7},
       pages={655\ndash 696},
         url={https://doi.org/10.1016/S0021-7824(01)01251-X},
      review={\MR{1968337}},
}

\bib{reywei}{article}{
      author={Rey, Olivier},
      author={Wei, Juncheng},
       title={Blowing up solutions for an elliptic {N}eumann problem with sub-
  or supercritical nonlinearity. {I}. {$N=3$}},
        date={2004},
        ISSN={0022-1236},
     journal={J. Funct. Anal.},
      volume={212},
      number={2},
       pages={472\ndash 499},
         url={http://dx.doi.org/10.1016/j.jfa.2003.06.006},
      review={\MR{2064935}},
}

\bib{MR2497911}{incollection}{
      author={Wei, J.},
       title={Existence and stability of spikes for the {G}ierer-{M}einhardt
  system},
        date={2008},
   booktitle={Handbook of differential equations: stationary partial
  differential equations. {V}ol. {V}},
      series={Handb. Differ. Equ.},
   publisher={Elsevier/North-Holland, Amsterdam},
       pages={487\ndash 585},
         url={https://doi.org/10.1016/S1874-5733(08)80013-7},
      review={\MR{2497911}},
}

\end{biblist}
\end{bibdiv}
\end{document}